\documentclass[reqno,preprint]{elsarticle}
\usepackage{amsmath,amsfonts,amssymb,amsthm}
\def\R{\mathbb R}
\def\C{\mathbb C}
\def\S{\mathbb S}
\def\D{\mathrm D}
\def\j{\mathrm i}
\def\d{\,\mathrm d}
\DeclareMathOperator{\diag}{diag}
\let\Im\relax
\DeclareMathOperator{\Im}{Im}

\DeclareMathOperator{\spec}{spec}
\newtheorem{thm}{Theorem}[section]
\newtheorem{prop}[thm]{Proposition}
\newtheorem{cor}[thm]{Corollary}

\theoremstyle{remark}
\newtheorem{rem}[thm]{Remark}
\numberwithin{equation}{section}
\begin{document}
\title{On $t$-dependent hyperbolic systems. Part 2.}
\author{Jens Wirth}
\ead{jens.wirth@mathematik.uni-stuttgart.de}
\address{Institute of Analysis, Dynamics and Modelling; Department of Mathematics; University of Stuttgart; Pfaffenwaldring 57; 70569 Stuttgart; GERMANY}
\begin{abstract}
We consider hyperbolic equations with time-dependent coefficients and develop an abstract framework to derive the asymptotic behaviour of the representation of solutions for large times. We are dealing with generic situations where the large-time asymptotics is of hyperbolic type. 

Our approach is based on diagonalisation procedures combined with asymptotic integration arguments.
\end{abstract}
\maketitle

\section{Introduction}
In a first part \cite{RW11} of the present paper written jointly with M.~Ruzhansky  we developed tools to deal with generic hyperbolic systems with $t$-dependent coefficients based on natural assumptions and to derive energy and dispersive type estimates, but mainly focussed on the treatment of high frequencies and on estimating resulting Fourier integrals. The treatment of small frequencies in \cite{RW11} was quite basic and with assumptions tailored towards a simple and rough estimate. Here we will complement and extend this paper by giving an alternative and more systematic approach to the treatment of small frequencies.

In \cite{NW15} a new method was developed by the author jointly with W.~Nunes do Nascimento to describe the large-time behaviour of solutions to hyperbolic equations with time-dependent coefficients and to characterise it in terms of large-time principal symbols, see Sections~\ref{ssec:2.1.1}  and \ref{ssec:2.1.4} for an explanation. This technique will be developed further in this paper and in particular applied to a broad range of hyperbolic systems and equations with time-dependent coefficients. 

Before giving model cases and admissible examples for our treatment, we will recall a bit of the history of the underlying problem and the mathematical background.
Dispersive type estimates are an essential tool to study non-linear equations and to prove stability and scattering results, but they are also interesting on their own right due to their relation to Fourier restriction theorems and to decay properties of Fourier transforms of surface carried measures. For the wave equation, dispersive estimates go back to the work P.~Brenner \cite{Bre75} and R.S.~Strichartz \cite{Str70}. Higher order problems were first considered by M.~Sugimoto \cite{Sug94,Sug96} based on refined stationary phase estimates for slowness surfaces with critical points. This was the starting point for a systematic treatment of decay properties of solutions to general scalar higher order equations with constant coefficients and to hyperbolic systems with constant coefficients as carried out by M.~Ruzhansky and J.G.~Smith \cite{RS10}. Equations of second order and with time-dependent coefficients have been intensively studied in the literature, examples include the work of M.~Reissig, K.~Yagdjian and J.G.~Smith \cite{RY00, RY00b, RS05} or the authors own work on dissipative wave equations \cite{Wir06, HW08, Wir10}. There are also known results on hyperbolic systems with time-dependent coefficients or equations of higher order, see e.g., the work of M.~d'Abbico, S.~Lucente,
G.~Taglialatela, M.~Reissig and M.~Ebert \cite{dALT09, dAR11, dAE12}, and approaches based on asymptotic integration techniques by T.~Matsuyama and M.~Ruzhansky \cite{MR10, MR13, MR13b} with applications to the well-posedness of Kirchhoff equations. 

All these approaches are tailored to specific situations. This has the advantage of obtaining very precise conditions on coefficients and gives also the possibility of slight improvements of decay rates and / or assumptions on initial data. What we aim for in this paper and to contrast these earlier investigations is a general treatment of most of these cases allowing for a more structural understanding of the properties of such equations.

The structure of this paper is as follows. First we will explain several model cases and examples which can be treated by our methods. This will be done in Section~\ref{sec2} combined with the precise description of our basic assumptions. After this, Section~\ref{sec3} will deal with the construction of representation of solutions and Section~\ref{sec4} provides the resulting energy and dispersive type estimates. Finally, Section~\ref{sec5} comes back to the examples and explains what kind of results one can obtain in such particular circumstances. 
 
In order to simplify notation, we write $f\lesssim g$ for two functions $f$ and $g$ if there exists a constant $C$ such that $f\le C g$ uniform in the arguments. We also write $f\asymp g$ if $f\lesssim g$ and $g\lesssim f$. If $k$ is a parameter or one of the arguments, we write $f\lesssim_k g$ to emphasise that the constant depends on the parameter $k$. We use a similar notation for asymptotic inequalities, $f(t)\lesssim g(t),\;t\to\infty$, means that there exists a number $t_0$ such that $f\lesssim g$ on $\{t \; :\;t\ge t_0\}$.

\section{Model cases and basic assumptions}\label{sec2}

\subsection{Notation and motivating examples} We recall notation used in \cite{RW11} and \cite{RW14}. For $\ell\in\R$ we denote by 
\begin{equation}
    \mathcal T\{\ell\} = \left\{ f \in C^\infty(\R_+) \;\bigg|\; \left| \partial_t^k f(t)\right| \lesssim_k \left( \frac1{1+t}\right)^{\ell+k} \right\}
\end{equation} 
the set of smooth functions satisfying a suitable symbol like behaviour. 

We will treat the $\mathcal T$-classes as classes of coefficient functions for hyperbolic problems. To motivate the later treatment we consider first some examples. They all fit into the general scheme 
\begin{equation}\label{eq:hyp-sys-1}
   \D_t U = A(t,\D_x) U,\qquad U(0,\cdot)=U_0,
\end{equation}
of systems with time-dependent matrix valued Fourier multipliers $A(t,\D_x)$ satisfying uniform strict hyperbolicity assumptions combined a suitable estimates for the $t$-dependence. Precise conditions are given in Section~\ref{ssec:2.1.4}.

\subsubsection{Differential hyperbolic systems}\label{ssec:2.1.1}
For coefficient matrices $A_j\in\mathcal T\{0\}\otimes\C^{d\times d}$ and $B\in\mathcal T\{1\}\otimes\C^{d\times d}$ we consider the Cauchy problem
\begin{equation}\label{eq:d-hyp-sys-1}
    \D_t U = \sum_{j=1}^n A_j(t) \D_{x_j} U + B(t) U,\qquad\qquad U(0,\cdot)=U_0,
\end{equation}
where $\D=-\j\partial$, $x\in \R^n$ and $t\ge0$. In order for this system to be {\em uniformly strictly hyperbolic} we assume that the matrix-valued {\em hyperbolic principal symbol} 
\begin{equation}\label{eq:A-pol}
    A(t,\xi) = \sum_{j=1}^n A_j(t) \xi_j
\end{equation}
has for $\xi\in\S^{n-1}=\{\xi\in\R^n\;:\; |\xi|=1\}$ and $t\ge0$ uniformly distinct real eigenvalues. We assume further the following technical-looking condition on the existence of a uniformly bounded and invertible diagonaliser $M(t,\xi)$ of $A(t,\xi)$ such that
\begin{multline}\label{eq:GECL-hyp-1}
   \Im \int_s^t \diag(M^{-1}(\tau,\xi) B(\tau,\xi) M(\tau,\xi) + (\D_\tau M^{-1}(\tau,\xi))M(\tau,\xi) ) \d\tau \\ \asymp  ( \log(1+ t) - \log(1+s) )\mathrm I
\end{multline}
holds true uniformly in $t$, $s$ and $\xi$. It turns out that the leading terms of the integrand are independent of the particular choice of the diagonaliser $M$ and should be considered as a {\em hyperbolic subprincipal symbol} of the system.   
It is shown in \cite[Thm. 4.4]{RW14} that the boundedness of the left-hand side is equivalent to a form of generalised energy conservation for high frequencies. See also \cite{HW09} for a simpler case. The condition can be reduced to spectral conditions on the matrix $B$ if the all the $A_j$ are self-adjoint. 

Both conditions are sufficient to give a full description of the behaviour of the spatial Fourier transform $\widehat U(t,\xi)$ of solutions for $(1+t)|\xi| \gtrsim 1$. The large-time behaviour of solutions in the remaining part of the extended phase depends on a second principal symbol. Multiplying equation \eqref{eq:d-hyp-sys-1} by $t$ gives the Fuchs-type equation
\begin{equation}
       t \D_t U = \sum_{j=1}^n  tA_j(t)  \D_{x_j} U + tB(t) U,
\end{equation}
and we assume that 
\begin{equation}\label{eq:B-cond-hyp1}
    \lim_{t\to\infty} tB(t) = B_\infty \in \C^{d\times d},\qquad \int_1^\infty \| \tau B(\tau) - B_\infty \|^\sigma \frac{\d\tau}{\tau} < \infty
\end{equation}
for some diagonalisable matrix $B_\infty$ and a number $\sigma\ge 1$. The eigenvalues of $B_\infty$ determine the large-time behaviour of solutions for $(1+t)|\xi|\lesssim 1$. We refer to the matrix $B_\infty$ (or rather to the pair consisting of the eigenvalues and the eigenvectors, see Section~\ref{ssec:2.1.4}) as the 
{\em large-time principal symbol}  of the hyperbolic system \eqref{eq:d-hyp-sys-1}. Together with the constants in the estimates \eqref{eq:GECL-hyp-1} it determines the large-time behaviour for all frequencies.

\subsubsection{Hyperbolic systems of higher order}  Hyperbolic systems of higher order arise for example when studying wave phenomena in elasticity. We consider only a particular model
\begin{align}\label{eq:hyp-sys-2}
    \D_t^2 U &= \sum_{i,j=1}^n A_{i,j}(t) \D_{x_i}\D_{x_j} U + B_0(t) \D_t U + \sum_{j=1}^n B_j(t) \D_{x_j} U + C(t) U,\\ U(0,\cdot)&=U_0, \qquad \D_t U(0,\cdot) = U_1,
\end{align}
with coefficients $A_{i,j}\in \mathcal T \{ 0\}\otimes \C^{d\times d}$, $B_j\in\mathcal T\{1\}\otimes\C^{d\times d}$ and $C\in\mathcal T\{2\}\otimes\C^{d\times d}$ subject to natural conditions. First, we assume that the matrix-valued function
\begin{equation}\label{eq:2.9}
    A(t,\xi) = \sum_{i,j=1}^n A_{i,j}(t) \xi_i\xi_j
\end{equation}
has distinct and strictly positive real eigenvalues uniform with respect to $t\ge 0$ and $\xi\in\S^{n-1}$. 
We define the auxiliary function (later we need a smoothed version of this)
\begin{equation}
 h(t,\xi)   =\begin{cases}|\xi|,\qquad & (1+t)|\xi|\ge N,\\
\frac N{1+t}, & (1+t)|\xi|\le N. \end{cases}
 \end{equation}
It allows to rewrite the second-order system \eqref{eq:hyp-sys-2} as system of first order in 
 $V = \big( \D_t U,h(t,\D_x) U\big)^\top$,
\begin{equation}\label{eq:2.11}
   \D_t V = \begin{pmatrix} B_0(t) & \left(A(t,\D_x) + \sum_{j=1}^n B_j(t) \D_{x_j} +   C(t)\right) h(t,\D_x) ^{-1}  \\
  h(t,\D_x) \mathrm I & (\D_t h(t,\D_x)) h(t,\D_x) ^{-1}  \mathrm I \end{pmatrix} V.
\end{equation}
This system of first order is  pseudo-differential. Assumption \eqref{eq:2.9} implies its uniform strict hyperbolicity, the hyperbolic subprincipal symbol involves $A(t,\xi)$ in combination with the coefficients $B_j(t)$, while the large-time principal symbol is determined by the matrices $B_0(t)$ and $C(t)$.

\subsubsection{Scalar higher order equations} Scalar higher order equations can also be reduced to pseudo-differential first order systems. We consider
\begin{equation}
     \D_t^m u + \sum_{k=0}^{m-1} \sum_{|\alpha|\le m-k} a_{k,\alpha}(t) \D_t^k \D_x^\alpha u = 0
\end{equation}
endowed with suitable initial conditions
\begin{equation}
   \D_t^k u(0,\cdot) = u_k,\qquad k=0,1,\ldots, m-1,
\end{equation}
and for coefficient functions $a_{k,\alpha}\in\mathcal T\{m-k-|\alpha|\}$. Assuming that the principal part is uniformly strictly hyperbolic, i.e., that the polynomial
\begin{equation}\label{eq:2:ev-eq}
   \lambda^m + \sum_{k=0}^{m-1} \sum_{|\alpha|=m-k} a_{k,\alpha}(t) \lambda^k \xi^\alpha =0
\end{equation}
has distinct real roots $\lambda_j(t,\xi)$ uniform in $t\ge0$ and $\xi\in\S^{n-1}$ and that a suitable replacement of \eqref{eq:GECL-hyp-1} and of \eqref{eq:B-cond-hyp1}
holds true, we can provide an asymptotic construction of solutions. The large-time principal symbol involves only the asymptotic behaviour of the terms $a_{k,0}(t)$.

\subsection{General strategy} To cover all the above examples it is sufficient to consider first order systems with general Fourier multipliers as coefficients. Our strategy is as follows. We consider transformations of such systems to make them suitable for asymptotic integration arguments. Such transformations are done in Fourier space / extended phase space and depend heavily on the interplay between time and spatial frequency. For large frequencies we apply a standard hyperbolic theory diagonalising the full symbol of the operator and constructing a WKB representation of the fundamental solution. The core of this strategy was given by K.~Yagdjian in \cite{YagBook} and further developed in \cite{RY00,RY00b,RS05} and \cite{RW11,RW14} for the treatment of the large-time behaviour of uniformly strictly hyperbolic systems. Our approach for small frequencies extends \cite{NW15}. We rewrite the system in Fuchs type and apply asymptotic integration techniques going back to Levinson or Hartmann--Wintner. 

\subsection{Definitions: Zones, Symbol classes} We decompose the phase space $\R_+\times\R^n$ into zones
\begin{equation}
\begin{split}
   \mathcal Z_{\rm pd}(N) = \{ (t,\xi) \;:\; (1+t) |\xi|\le N \}, \\ \mathcal Z_{\rm hyp}(N)= \{ (t,\xi) \;:\; (1+t) |\xi|\ge N \}.
\end{split}
\end{equation}
As outlined above, our strategy depends heavily on the zone under consideration.
We denote by $t_\xi\ge0$ the function defined by $(1+t_\xi)|\xi|=N$, i.e., $t_\xi = N/|\xi|-1$ for $0<|\xi|\le N$ and $t_\xi=0$ for $|\xi|\ge N$. It parameterises the boundary between the zones. 

To formulate our main assumptions precisely and to carry out the diagonalisation procedures we need symbol classes. For this we follow \cite{RW11} and define
$\mathcal S\{m_1,m_2\}$ to be
\begin{equation}
\bigcup_N \left\{ a\in  C^\infty( \mathcal Z_{\rm hyp}(N) ) \;:\;   |\D_t^k \D_\xi^\alpha a(t,\xi) | \lesssim_{k,\alpha,N} |\xi|^{m_1-|\alpha|} \left( \frac{1}   {1+t}\right)^{m_2+k} \right\},
\end{equation} 
i.e., the symbolic estimates should be satisfied on some hyperbolic zone $\mathcal Z_{\rm hyp}(N)$ for $N$ chosen large enough. These symbol classes have natural embedding properties, most important for us is 
\begin{equation}
   \mathcal S \{m_1,m_2\} \subset \mathcal S\{m_1',m_2'\}
\end{equation}
whenever $m_1\le m_1'$ and $m_1+m_2 \ge m_1'+m_2'$, and behave well under multiplication and differentiation. We also define the residual class
\begin{equation}
   \mathcal H\{k\} = \bigcap_{m_1+m_2=k} \mathcal S\{m_1,m_2\}.
\end{equation}
It plays a role as remainder class when performing perfect diagonalisation within the hyperbolic zone $\mathcal Z_{\rm hyp}(N)$. Furthermore, symbols from the class $\mathcal S\{-1,2\}$ are uniformly integrable with respect to $t$ over $\mathcal Z_{\rm hyp}(N)$. For $a\in\mathcal S\{-1,2\}$ the symbol estimate implies for $|\xi|\le N$
\begin{equation}
    \int_{t_\xi}^\infty |a(t,\xi)| \d t \lesssim \frac{1}{|\xi| (1+t_\xi)} \lesssim 1
\end{equation}
and for $|\xi|\ge N$
\begin{equation}
    \int_{0}^\infty |a(t,\xi)| \d t \lesssim \frac{1}{|\xi|} \lesssim 1.
\end{equation}

\subsection{Main assumptions}\label{ssec:2.1.4} We consider the Cauchy problem for a system \eqref{eq:hyp-sys-1},
where the symbol  $A(t,\xi)$ of the Fourier multiplier $A(t,\D_x)$ satisfies the following assumptions:
\begin{description}
\item[(A1)] There exists a positively homogeneous symbol $A_{\rm hom}(t,\xi)\in\mathcal S\{1,0\}\otimes\C^{d\times d}$, i.e., we assume $A_{\rm hom}(t,\rho \xi) = \rho A_{\rm hom}(t,\xi)$ for $\rho>0$, such that
\begin{equation}
A-A_{\rm hom}\in \mathcal S\{0,1\}\otimes\C^{d\times d}
\end{equation}
and that the eigenvalues $\lambda_j(t,\xi)$, $j=1,\ldots, d$ of $A_{\rm hom}(t,\xi)$ are real and satisfy
\begin{equation}
     |\lambda_i(t,\xi) - \lambda_j(t,\xi)| \ge \delta |\xi|
\end{equation}
uniformly in $t\ge 0$ and for all $i\ne j$. The system \eqref{eq:hyp-sys-1} is uniformly strictly hyperbolic if this assumption is satisfied.
We refer to $A_{\rm hom}$ as the {\em hyperbolic principal symbol}.
\item[(A2)] Let $M\in\mathcal S\{0,0\}\otimes\C^{d\times d}$ be a diagonaliser of the hyperbolic principal symbol $A_{\rm hom}$ with $M^{-1}\in\mathcal S\{0,0\}\otimes\C^{d\times d}$ (which exists by (A1) as shown in \cite[Lem. 4.2]{RW14}) and let 
\begin{equation}
F_0 = \diag \big( M^{-1} (A-A_{\rm hom}) M + (\D_t M^{-1})M \big) \mod \mathcal S\{-1,2\} \otimes \C^{d\times d}.  
\end{equation}
We refer to $F_0\in \mathcal S\{0,1\}\otimes\C^{d\times d}$ as the {\em hyperbolic subprincipal symbol}.
Then we assume that for constants $\kappa_\pm\in\R$ and all $t\ge s \ge t_\xi$
\begin{equation}\label{eq:2.25}
  \kappa_+ \log\frac{1+t}{1+s} + C_+ \le  \Im \int_s^t  F_0(\tau,\xi) \d\tau  \le \kappa_- \log\frac{1+t}{1+s} + C_-
\end{equation} 
holds true with suitable $C_\pm\in\R$. 
\item[(A3)] 
There exists $\Lambda\in \mathcal T\{0\}\otimes\C^{d\times d}$ diagonal and  $\tilde M\in\mathcal T\{0\}\otimes\C^{d\times d}$  invertible with $\tilde M^{-1} \in \mathcal T\{0\}\otimes\C^{d\times d}$ such
that 
 \begin{equation}
   \tilde R =  t \tilde M^{-1} A \tilde M - \Lambda - (t\D_t \tilde M^{-1})\tilde M
 \end{equation}
 is small in the sense that
\begin{equation}
 \sup_{\xi\,:\, t_\xi\ge 1} \int_1^{t_\xi} \| \tilde R(t,\xi) \|^\sigma \frac{\d t}{t} <\infty
\end{equation}
holds true for some constant $\sigma\ge1$ and a zone constant $N$. We refer to to the pair $(\Lambda,\tilde M)$ as the {\em large time principal symbol}
and assume further that its diagonal entries $\Lambda=\diag(\mu_1,\ldots,\mu_d)$ satisfy 
\begin{itemize}
\item {\bf in the case $\sigma=1$} the weak dichotomy condition
\begin{equation}\label{eq:weak-dich} 
\begin{split}
& \limsup_{t\to\infty}  \Im \int_1^t \big( \mu_i(\tau) - \mu_j(\tau) \big)  \frac{\d\tau}{\tau} <\infty \\
 \qquad \text{or} \qquad &\\
 & \liminf_{t\to\infty}  \Im \int_1^t \big( \mu_i(\tau) - \mu_j(\tau) \big)  \frac{\d\tau}{\tau} > -\infty;
\end{split}
\end{equation}
\item and {\bf in the case $\sigma>1$} the strong dichotomy condition
\begin{equation}\label{eq:strong-dich} 
   |\Im (\mu_i(t) - \mu_j(t) )| \ge \delta >0,\qquad i\ne j.
\end{equation}
\end{itemize}
\item[(A4)] For sufficiently large $N$ the  the estimate
\begin{equation}
   \| \D_t^k \D_\xi^\alpha A(t,\xi) \| \lesssim_{k,\alpha,N} (1+t)^{-1-k+|\alpha|}
\end{equation}
holds true within $\mathcal Z_{\rm pd}(N)$ and for all $k$ and all multi-indices $\alpha$.
\end{description}

\begin{rem}
We remark that the constant $N$ in assumptions (A2), (A3) and (A4) depends on the estimate we have in mind, it is determined by the number of diagonalisation steps to be carried out within the hyperbolic zone and a sufficiently large $N$ ensures the invertibility of certain multipliers. If (A2) is valid in some hyperbolic zone, it is also valid in smaller hyperbolic zones with 
possibly smaller difference $\kappa_+-\kappa_-$. If (A3) is valid for some zone constant $N$ it is automatically valid for all zone constants $N$. This follows from the uniform boundedness of $\|\tilde R(t,\xi)\|$ on every strip $N \le (1+t)|\xi|\le \tilde N$ combined with
\begin{equation}
   \int_{t_\xi}^{\tilde t_\xi} \frac{\d\tau}{\tau} = \log \frac{\tilde t_\xi}{ t_\xi} = \log\frac{\tilde N}{N}. 
\end{equation}
For practical applications, see e.g.~in \cite{NW15}, assumption (A4) is satisfied for all choices of the zone constant $N$. 
\end{rem}

\begin{rem}
Assumption (A2) with $\kappa_-=\kappa_+=0$ reduces to the generalised energy conservation property (GECL) introduced in \cite{HW09} and also considered in \cite{RW14}. It turns out that in this particular case uniform lower and upper bounds for the energy of solutions in terms of the initial energy follow, provided the Fourier support of the initial data does not contain $\xi=0$. On the other hand, assumption (A1) implies that (A2) is valid for some numbers $\kappa_\pm$ and we introduce this assumption mainly to fix these important constants.
\end{rem}

\begin{rem}
If we choose in (A2) a different diagonaliser $\widetilde M \in \mathcal S\{0,0\}\otimes \C^{d\times d}$ of $A_{\rm hom}$ with $\widetilde M^{-1} \in \mathcal S\{0,0\}\otimes \C^{d\times d}$ such that 
$M^{-1}A_{\rm hom}M = \widetilde M^{-1} A_{\rm hom} \widetilde M$, then both are related by a smooth diagonal matrix $H$ with uniformly bounded entries and uniformly bounded inverse 
\begin{equation}
  \widetilde M(t,\xi)  = M(t,\xi) H(t,\xi).
\end{equation}
In consequence $F_0$ is changed to
\begin{equation}
   \widetilde F_0(t,\xi) = F_0 (t,\xi) + \D_t  \log H(t,\xi)  
\end{equation}
and estimate \eqref{eq:2.25} holds true with the same constants $\kappa_\pm$.
\end{rem}

\begin{rem}
For some applications it is convenient to use the following stronger form of assumption (A3). If there exists a diagonalisable matrix $A_\infty\in\C^{d\times d}$ such that  
\begin{equation}
   \sup_{\xi \,:\, t_\xi\ge 1}  \int_1^{t_\xi} \| t A(t,\xi) - A_\infty \|^\sigma \frac{\d t}{t} < \infty
\end{equation}
holds true with some exponent $\sigma\ge1$ and for sufficiently large $N$, then (A3) follows with the (constant) diagonaliser of $A_\infty$ as $\tilde M$ and the resulting diagonal matrix as $\Lambda$. The weak dichotomy condition is automatically satisfied. In order to guarantee the strong dichotomy condition we have to assume that the eigenvalues of $A_\infty$ are simple.
\end{rem}

\begin{rem}
It is possible to generalise Assumption (A3) to matrices $\Lambda$ and $\tilde M$ which also depend on directions $\xi/|\xi|$. This would not change any of the results in Section~\ref{sec:3.1}, except for complicating notation. As this is not needed for our applications to pseudodifferential hyperbolic systems of first order, we omit this. 
\end{rem}

\section{Construction of fundamental solutions}\label{sec3}
We concentrate mainly on the construction in $\mathcal Z_{\rm pd}(N)$. It follows \cite{NW15} and uses in particular the tools provided in the appendix to that paper. For the hyperbolic zone
$\mathcal Z_{\rm hyp}(N)$ we will remain a bit sketchy and refer to \cite{RW11} and \cite{RW14} for the detailed procedure to be carried out. 

If $U$ solves \eqref{eq:hyp-sys-1}, its spatial Fourier transform $\widehat U(t,\xi)$ satisfies the parameter dependent ordinary differential equation
\begin{equation}
    \D_t \widehat U(t,\xi) = A(t,\xi) \widehat U(t,\xi),\qquad \widehat U(0,\xi)=\widehat U_0(\xi),
\end{equation}
and can thus be represented in terms of the fundamental solution $\mathcal E(t,s,\xi)$, i.e., the solution to the matrix-valued problem
\begin{equation}
    \D_t \mathcal E(t,s,\xi) = A(t,\xi) \mathcal E(t,s,\xi),\qquad \mathcal E(s,s,\xi)=\mathrm I\in\C^{d\times d},
\end{equation}
as $\widehat U(t,\xi) = \mathcal E(t,0,\xi)\widehat U_0$. Our aim is to provide an asymptotic construction of this fundamental solution and to derive estimates from that.

\subsection{Asymptotic integration for small frequencies}\label{sec:3.1}
The large time principal symbol $\Lambda(t) = \diag(\mu_1(t),\ldots,\mu_d(t))\in\mathcal T\{0\}\otimes\C^{d\times d}$ given by assumption (A3) determines the large-time behaviour of solutions to 
\begin{equation}\label{eq:3.5}
    t\D_t \widehat U(t,\xi) = tA(t,\xi) \widehat U(t,\xi)
\end{equation}
within $\mathcal Z_{\rm pd}(N)$ as $t\to\infty$.  We assume further for the moment that $\sigma=1$. Then the transformation $\widehat U^{(0)}(t,\xi) =  \tilde M^{-1}(t) \widehat U(t,\xi)$ yields 
\begin{equation}\label{eq:3.6}
  t\D_t \widehat U^{(0)} = \Lambda(t)  \widehat U^{(0)} + \tilde R(t,\xi) \widehat U^{(0)}(t,\xi)
\end{equation}
and Levinsons theorem in the form \cite[Thm. A.1]{NW15} allows to construct a fundamental system of solutions of the form
\begin{equation}\label{eq:3:V_j-as}
  (v_j(t)  + o(1)) \exp\left( \mathrm i \int_1^t \mu_j(\tau)\frac{\d\tau}{\tau}\right) ,\qquad j=1,\ldots, d,
\end{equation}
with $M(t) = (v_1(t) | v_2(t) | \cdots |v_d(t))\in\mathcal T\{0\}\otimes\C^{d\times d}$ having the columns $v_j(t)$. As these vectors and the $\mu_j$ are uniformly bounded, 
the exponential term shows polynomial behaviour. This behaviour can be estimated in terms of the number 
\begin{equation}\label{eq:mu-def}
   \mu = \min_{j=1,\ldots,d} \liminf_{t\to\infty} \frac{\Im \int_1^t \mu_j(\tau) \frac{\d\tau}{\tau}} {\log t}. 
\end{equation}
Up to a small loss in the order this number determines the estimates also in the case $\sigma>1$.

\begin{thm}\label{thm:3.1}
Assume (A3) and let $\mu$ be defined by \eqref{eq:mu-def}.  Then uniform in $(s,\xi),(t,\xi)\in\mathcal Z_{\rm pd}(N)$ with $s\le t$ the estimate
\begin{equation}
    \| \mathcal E(t,s,\xi) \| \lesssim_{\epsilon,N} \left(\frac{1+t}{1+s}\right)^{-\mu+\epsilon}
\end{equation}
holds true with $\epsilon=0$ in the case $\sigma=1$ and for arbitrary $\epsilon>0$ in the case $\sigma>1$.
\end{thm}

\begin{proof}
The case $\sigma=1$ just follows \cite[Thm. A.1]{NW15} and the two remarks after it. The transformed unknown $\widehat U^{(0)}=\tilde M^{-1}\widehat U$ solves \eqref{eq:3.6}
with diagonal $\Lambda(t)$ satisfying the (weak) dichotomy condition \eqref{eq:weak-dich} and remainder $\tilde R(t,\xi)$ satisfying the integrability condition 
\begin{equation}\label{eq:3.7}
	\sup_{\xi: t_\xi\ge 1} \int_1^{t_\xi} \| \tilde R(t,\xi)\|\frac{ \d t  } t<\infty.
\end{equation}
Therefore, we find asymptotic solutions to \eqref{eq:3.6} of the form
\begin{equation}
  \big(e_j +  o(1)\big) \exp\left(\mathrm i \int_1^t \mu_j(\tau)\frac{\d\tau}{\tau}\right),\qquad t\to\infty
\end{equation}
uniform in $\xi$ with $(t,\xi)\in\mathcal Z_{\rm pd}(N)$ and with $e_j$ the $j$-th basis vector of $\C^d$. Transforming back gives a fundamental system $V_j(t,\xi)$ of \eqref{eq:3.5}
satisfying \eqref{eq:3:V_j-as}. The Wronskian of this fundamental system satisfies
\begin{equation}
   \mathcal W_{V_1,\ldots V_d}(t)  = \det\big( V_1(t,\xi) |  \cdots | V_d(t,\xi)\big) = \exp\left( \mathrm i \int_1^t \sum_{j=1}^d \mu_j(\tau) \frac{\d\tau}{\tau} \right) 
\end{equation}
by Liouville theorem combined with the above asymptotics. Hence, the fundamental matrix $\mathcal E(t,1,\xi)$ given as
\begin{equation}\label{eq:3.11}
  \mathcal E(t,1,\xi) = \big( V_1(t,\xi) | \cdots | V_d(t,\xi)\big) \big( V_1(1,\xi) | \cdots | V_d(1,\xi)\big)^{-1} 
\end{equation}
can be estimated by applying Hadamard's inequality to Cramer's rule
\begin{equation}
   \| \mathcal E(t,1,\xi)\| \lesssim_{N} t^{-\mu}. 
\end{equation}
Combined with the scaling argument from  \cite[Rem. A.2]{NW15} 
\begin{equation}
   \| \mathcal E(t,s,\xi) \| \lesssim_N \left(\frac{1+t}{1+s}\right)^{-\mu}
\end{equation}
follows. For the case $\sigma>1$ equation \eqref{eq:3.7} holds true with the exponent $\sigma$. In order to improve integrability we apply the Hartman--Wintner theorem in the form \cite[Thm. A.2]{NW15}. This transforms the system \eqref{eq:3.6} into a new system 
\begin{equation}%\label{eq:3.6}
\begin{split}
t\D_t \widehat U^{(1)}(t,\xi)& = \big(\Lambda_1 (t,\xi)  +  \tilde R_1(t,\xi) \big) \widehat U^{(1)}(t,\xi),\\
 \Lambda_1(t,\xi) &=  \Lambda(t)+\diag \tilde R(t,\xi)
\end{split}
\end{equation}
with 
\begin{equation}
	\sup_{\xi:t_\xi\ge1} \int_1^{t_\xi} \| \tilde R_1(t,\xi)\|^{\max\{\sigma/2,1\}} \frac{ \d t  } t<\infty
\end{equation}
and iteratively yields after $k$ steps (with $k$ chosen large enough to guarantee $\sigma/2^k<1$) a system of Levinson form. This can then be dealt with as in the first case. Because the new diagonal matrix $\Lambda_k(t,\xi)$ differs from $\Lambda$ by terms satisfying the $\sigma$-integrability condition, the asymptotic behaviour changes slightly. Let for this $r(t,\xi)$ be real-valued and satisfy the $\sigma$-integrability condition 
\begin{equation}
\sup_{\xi : t_\xi\ge 1} \int_1^{t_\xi} |r(t,\xi)|^\sigma \frac{\d t}t <\infty
\end{equation} 
and let $\mu$ be a real number. Then by H\"older inequality and with $\sigma'$ dual to $\sigma$
\begin{equation}
\begin{split}
\exp\left( \int_1^{t}( -\mu + r(\tau,\xi)) \frac{\d \tau}\tau  \right) &= t^{-\mu} \exp\left( \int_1^{t} r(\tau,\xi) \frac{\d\tau}{\tau} \right)\\
&\le t^{-\mu} \exp\left( \|r\|_{L^\sigma}   \left(\int_1^t \frac{\d\tau}{\tau} \right)^{1/\sigma'} \right) \\ 
& \le t^{-\mu} \exp\left( \|r\|_{L^\sigma} (\log t)^{1/\sigma'}\right) \lesssim_\epsilon t^{-\mu+\epsilon}
\end{split}
\end{equation}
for any $\epsilon>0$.
\end{proof}

\subsection{Estimating derivatives} Assumption (A4) allows to estimate derivatives of $\mathcal E(t_\xi,0,\xi)$ with respect to the frequency variable $\xi$. This will be essential to prove dispersive type estimates.

\begin{thm}
Assume (A3) and (A4). Then the estimate
\begin{equation}
   \| \D_\xi^\alpha \mathcal E(t_\xi,0,\xi)\| \lesssim_{\alpha,\epsilon,N} (1+t_\xi)^{-\mu+\epsilon} |\xi|^{-|\alpha|},\qquad |\xi|\le N
\end{equation}
holds true for all multi-indices $\alpha$ and with arbitrary $\epsilon>0$ for $\sigma>1$ and with $\epsilon=0$ for $\sigma=1$.
\end{thm}
\begin{proof}
We consider the fundamental solution $\mathcal E(t,0,\xi)$ for $(t,\xi)\in\mathcal Z_{\rm pd}(N)$. Its derivatives with respect to $t$ are easily estimated by the equation, indeed
from $\D_t \mathcal E = A(t,\xi) \mathcal E$ we observe in combination with Theorem~\ref{thm:3.1} that 
\begin{equation}
 \|  \D_t \mathcal E(t,0,\xi) \| \le \|A(t,\xi) \| \, \|\mathcal E(t,0,\xi) \| \lesssim_{\epsilon,N} (1+t)^{-\mu+\epsilon-1}.
\end{equation}
Higher order $t$-derivatives are estimated recursively taking into account the estimate of assumption (A4). For $\xi$-derivatives we formally differentiate the equation satisfied by $\mathcal E$ and obtain 
\begin{equation}\label{eq:3.14}
   \D_t \D_\xi^\alpha \mathcal E(t,0,\xi) = A(t,\xi) \D_\xi^\alpha \mathcal E(t,0,\xi) + \sum_{\beta<\alpha} \binom\alpha\beta\big( \D_\xi^\beta A(t,\xi)\big) \big(\D_\xi^{\alpha-\beta} 
   \mathcal E(t,0,\xi)\big)
\end{equation}
such that by Duhamel formula (and using $\D_\xi^\alpha \mathcal E(0,0,\xi)=0$ for $|\alpha|\ge1$)
\begin{equation}\label{eq:3.15}
   \D_\xi^\alpha \mathcal E(t,0,\xi) = \int_0^t \mathcal E(t,s,\xi) R_\alpha(s,\xi) \d s,
\end{equation}
where $R_\alpha(t,\xi)$ denotes the sum in \eqref{eq:3.14}. By induction we prove 
\begin{equation}\label{eq:3.16}
   \|\D_\xi^\alpha \mathcal E(t,0,\xi) \| \lesssim_{\alpha,\epsilon,N} (1+t)^{-\mu+\epsilon+|\alpha|}. 
\end{equation}
This was already shown for $|\alpha|=0$, it suffices to give the induction step. Let $|\alpha|\ge1$. Assuming \eqref{eq:3.16} for all multi-indices $\beta<\alpha$ yields in combination with (A4) the estimate
\begin{equation}
    \| R_\alpha(t,\xi)\| \lesssim_{\alpha,\epsilon,N} (1+t)^{-\mu+\epsilon+|\alpha|-1}.
\end{equation}
Hence by \eqref{eq:3.15} in combination with Theorem~\ref{thm:3.1} we obtain
\begin{equation}
\begin{split}
   \|\D_\xi^\alpha \mathcal E(t,0,\xi)\| &\lesssim_{\alpha,\epsilon,N} \int_0^t  \left(\frac{1+t}{1+s}\right)^{-\mu+\epsilon} (1+s)^{-\mu+\epsilon+|\alpha|-1}\d s \\&
   \lesssim_{\alpha,\epsilon,N}(1+t)^{-\mu+\epsilon} \int_0^t (1+s)^{|\alpha|-1} \d s 
\end{split}
\end{equation}
and thus the \eqref{eq:3.16} holds true for $\alpha$. Finally, mixed derivatives are easily estimated from differentiating \eqref{eq:3.15}. This yields
\begin{equation}
    \| \D_t^k \D_\xi^\alpha \mathcal E(t,0,\xi)\| \lesssim_{k,\alpha,\epsilon,N} (1+t)^{-\mu+\epsilon-k+|\alpha|} ,\qquad |\xi|\le N,
\end{equation}
and combined with the estimate
\begin{equation}\label{eq:3:txi-der-est}
   |\D_\xi^\alpha (1+ t_\xi) | \lesssim_\alpha |\xi|^{-1-|\alpha|},\qquad |\xi|\le N,
\end{equation}
the desired statement follows.
\end{proof}

\subsection{The diagonalisation scheme for high frequencies} We recall some facts about transformations to be applied within the hyperbolic zone $\mathcal Z_{\rm hyp}(N)$ in order to make the system suitable for asymptotic integration and WKB analysis. They are merely standard and follow \cite{RW11}, \cite{RW14}. They are entirely based upon assumption (A1). First, we recall \cite[Lem.~4.2]{RW14}.

\begin{prop}\label{prop:3.3}
\begin{enumerate}
\item
The eigenvalues $\lambda_j(t,\xi)$ of the hyperbolic principal part $A_{\rm hom}(t,\xi)$ satisfy  
\begin{equation}
\lambda_j\in \mathcal S\{1,0\}
\end{equation}
together with 
\begin{equation}
(\lambda_i-\lambda_j)^{-1} \in \mathcal S\{-1,0\},\qquad i\ne j.
\end{equation}
\item
There exists a matrix-valued symbol $M\in\mathcal S\{0,0\}\otimes\C^{d\times d}$, homogeneous in the sense that $M(t,\rho\xi)=M(t,\xi)$ for $\rho>1$,  with uniformly bounded inverse
$M^{-1}\in\mathcal S\{0,0\}\otimes\C^{d \times d}$ such that 
\begin{equation}
    M^{-1}(t,\xi) A_{\rm hom}(t,\xi) M(t,\xi) = \diag(\lambda_1(t,\xi),\ldots \lambda_d(t,\xi) )
\end{equation}
within $\mathcal Z_{\rm hyp}(N)$.
\end{enumerate}
\end{prop}

This proposition allows for a first transformation of the system. If the vector-valued function $U$ solves \eqref{eq:hyp-sys-1}, the function $V^{(0)}(t,\xi) = M^{-1}(t,\xi)\widehat U(t,\xi)$ solves
\begin{equation}
   \D_t V^{(0)}(t,\xi) =\big( \mathcal D(t,\xi) + R_0(t,\xi) \big) V^{(0)}(t,\xi)
\end{equation}
with diagonal main part
\begin{equation}
   \mathcal D = \diag\big( \lambda_1,\ldots,\lambda_d\big) \in \mathcal S\{1,0\}\otimes\C^{d\times d}
\end{equation}
and remainder term
\begin{equation}
   R_0 = M^{-1} (A-A_{\rm hom}) M + (\D_t M^{-1})M \in \mathcal S\{0,1\}\otimes\C^{d\times d}.
\end{equation}
Note that assumption (A2) is related to the diagonal part of that symbol, $F_0 = \diag R_0$. Choosing the hyperbolic zone small enough allows to perform further transformations to improve the remainder. The first step is as follows, cf. Section~4.2 from \cite{RW14}.

\begin{prop}\label{prop:3.4}
There exists a symbol $N_1\in\mathcal S\{0,0\}\otimes\C^{d\times d}$ satisfying $N_1^{-1} \in \mathcal S\{0,0\}\otimes\C^{d\times d}$ for a sufficiently large zone constant $N$ 
such that the operator-identity
\begin{equation}
    \big( \D_t - \mathcal D - R_0\big) N_1 = N_1  \big( \D_t - \mathcal D - F_0 - R_1\big) 
\end{equation}
holds true with $F_0= \diag R_0\in \mathcal S\{1,0\}\otimes\C^{d\times d}$ and $R_1\in\mathcal S\{-1,2\}\otimes\C^{d\times d}$. 
Furthermore,  $N_1-\mathrm I \in \mathcal S\{-1,1\}\otimes\C^{d\times d}$.
\end{prop}

This can be iteratively improved. The essence is given in the following proposition also taken from \cite{RW14}. The proof is standard and we omit it here.

\begin{prop}
For any number $k > 1$ there exists a matrix-valued symbol $N_k\in\mathcal S\{0,0\}\otimes\C^{d\times d}$ satisfying $N_k^{-1} \in \mathcal S\{0,0\}\otimes\C^{d\times d}$ for a sufficiently large zone constant $N$ 
such that the operator-identity
\begin{equation}
    \big( \D_t - \mathcal D - R_0\big) N_k = N_k  \big( \D_t - \mathcal D - F_{k-1} - R_k\big) 
\end{equation}
holds true with diagonal $F_{k-1}\in \mathcal S\{0,1\}\times\C^{d\times d}$ satisfying $F_{k-1}-F_{k-2}\in\mathcal S\{1-k,k\}\otimes \C^{d\times d}$ and $R_k\in\mathcal S\{-k,k+1\}\otimes\C^{d\times d}$. Furthermore $N_k- N_{k-1} \in \mathcal S\{-k,k\}\otimes \C^{d\times d}$.
\end{prop}

This transformation is sufficient to provide first uniform bounds on the fundamental solution $\mathcal E(t,s,\xi)$ within the hyperbolic zone $\mathcal Z_{\rm hyp}(N)$.
We briefly collect the key estimates for the symbols obtained after transformation. We fix the number $k$ of diagonalisation steps and in consequence also the zone constant $N$ and define  $t_\xi= \max\{ N/|\xi| -1, 0\}$ for this number $N$. Because $F_{k-1}-F_0$ is uniformly integrable over $\mathcal Z_{\rm hyp}(N)$ we know that
\begin{equation}\label{3.18}
 \left\| \exp\left(  \j \int_s^t F_{k-1}(\tau,\xi) \d\tau \right) \right\| \lesssim_{N,k} \left(\frac{1+ t}{1+s}\right)^{-\kappa_+},\qquad t\ge s,
\end{equation}
and
\begin{equation}\label{3.18'}
 \left\| \exp\left( - \j \int_s^t F_{k-1}(\tau,\xi) \d\tau \right) \right\| \lesssim_{N,k} \left(\frac{1+ t}{1+s}\right)^{\kappa_-},\qquad t\ge s,
\end{equation}
hold true with the numbers $\kappa_\pm$ from (A3). If $k> \kappa_- - \kappa_+$ the estimate
\begin{equation}\label{3.19}
   \int_{t_\xi}^\infty \|R_k(\tau,\xi)\| \tau^{\kappa_--\kappa_+} \d\tau \lesssim_{N,k} \begin{cases} |\xi|^{\kappa_+-\kappa_-},\qquad & |\xi|\ge N,\\ 1 %N^{\kappa_+-\kappa_-} 
   , & |\xi|\le N, \end{cases}
\end{equation}
holds also true. This allows to compensate the polynomial behaviour of the fundamental solution of the diagonal part of the transformed system and implies the following first theorem.

\begin{thm}\label{thm:3.5}
Assume (A1) and (A2). Then for any $k>\kappa_--\kappa_+ + 1$ we find a zone constant $N$ such that within $\mathcal Z_{\rm hyp}(N)$ the fundamental solution 
$\mathcal E(t,t_\xi,\xi)$ is representable as
\begin{equation}\label{eq:Eprodform}
   \mathcal E(t,t_\xi,\xi)  =  M(t,\xi) N_k(t,\xi) \widetilde{\mathcal E}_k(t,t_\xi,\xi) \mathcal Q_k(t,t_\xi,\xi) N_k^{-1}(t_\xi,\xi) M^{-1}(t_\xi,\xi)
\end{equation}
in terms of the matrices $M$ from Proposition~\ref{prop:3.3}, the matrices $N_k$ from Proposition~\ref{prop:3.4}, the diagonal matrices
\begin{equation}
    \widetilde{\mathcal E}_k(t,t_\xi,\xi) =  \exp\left(\j \int_{t_\xi}^t \big(\mathcal D(\tau,\xi)+ F_{k-1}(\tau,\xi)\big) \d\tau\right)
\end{equation}
satisfying the uniform bound
\begin{equation}
   \|   \widetilde{\mathcal E}_k(t,t_\xi,\xi) \| \lesssim_{N,k} \left(\frac{1+t}{1+t_\xi}\right)^{-\kappa_+},\qquad t\ge t_\xi,
\end{equation}
and matrices $\mathcal Q_k(t,s,\xi)$ uniformly bounded and uniformly invertible within $\mathcal Z_{\rm hyp}(N)$ satisfying the symbolic estimates
\begin{equation}
    \|\D_\xi^\alpha \mathcal Q_k(t,t_\xi,\xi)\| \lesssim_{N,k} |\xi|^\alpha,\qquad t\ge t_\xi
\end{equation}
for all multiindices $|\alpha|\le k-1+\kappa_+-\kappa_-$.
\end{thm}
Combining the uniform bounds of the diagonalisers in \eqref{eq:Eprodform} and using the estimate for $\mathcal Q_k$ with $|\alpha|=0$ we obtain in particular a uniform bound for the fundamental solution.
\begin{cor}
The estimate
\begin{equation}
   \|\mathcal E(t,t_\xi,\xi)\| \lesssim_N \left(\frac{1+t}{1+t_\xi}\right)^{-\kappa_+}
\end{equation}
holds true.
\end{cor}
\begin{proof}[Proof of Theorem~\ref{thm:3.5}]
We construct the fundamental solution in several steps. First, we solve the diagonal part $\D_t-\mathcal D-F_{k-1}$ of the transformed system. This can be done explicitly and yields the fundamental matrix
\begin{equation}
    \widetilde {\mathcal E}_k(t,s,\xi) = \exp\left(\j \int_s^t \big(\mathcal D(\tau,\xi)+ F_{k-1}(\tau,\xi)\big) \d\tau\right),
\end{equation}
which can be estimated together with its inverse by \eqref{3.18} and \eqref{3.18'}. 
To construct the fundamental solution $\widetilde{\mathcal E}(t,s,\xi)$ to the transformed system $\D_t -\mathcal D-F_{k-1}-R_k$
we use the ansatz $\widetilde{\mathcal E}(t,s,\xi) =    \widetilde {\mathcal E}_k(t,s,\xi)  \mathcal Q_k(t,s,\xi)$, which yields 
\begin{equation}
   \D_t \mathcal Q_k(t,s,\xi) = \big(  \widetilde {\mathcal E}_k(s,t,\xi)  R_k(t,\xi)  \widetilde {\mathcal E}_k(t,s,\xi) \big) \mathcal Q_k(t,s,\xi) = \widetilde{\mathcal R}_k(t,s,\xi) \mathcal Q_k(t,s,\xi)
\end{equation}
together with the initial condition $\mathcal Q_k(s,s,\xi)=\mathrm I$. The solution to this system can be represented in terms of the Peano--Baker series
\begin{align}\label{eq:PeanoBaker}
   \mathcal Q_k(t,s,\xi) &= \mathrm I + \sum_{\ell=1}^\infty \j^\ell \int_s^t \widetilde{\mathcal R}_k(t_1,s,\xi) \int_s^{t_1} \widetilde{\mathcal R}_k(t_2,s,\xi) \notag\\&\qquad\qquad \qquad\qquad \cdots 
   \int_s^{t_{\ell-1}} \widetilde{\mathcal R}_k(t_\ell,s,\xi) \d t_\ell \cdots \d t_2 \d t_1.
 \end{align}
Chosing $k$ large enough such that $k>\kappa_--\kappa_+$ implies uniform integrability of $\widetilde{\mathcal R}_k(t,t_\xi,\xi)$ over the hyperbolic zone and thus  the uniform bound
\begin{align}
   \| \mathcal Q_k(t,t_\xi,\xi) \|& \le \exp\left(\int_{t_\xi}^t \| \widetilde{\mathcal R}_k(\tau,t_\xi,\xi)\| \d\tau\right)\notag\\
   & \le \exp\left( \frac C{(1+t_\xi)^{\kappa_--\kappa_+}} \int_{t_\xi}^\infty
   \frac{\d\tau}{|\xi|^k (1+\tau)^{k+1-\kappa_-+\kappa_+}}  \right)
   \le \exp\left(\frac C{N^k}\right).
\end{align}
follows. It further implies the convergence $\mathcal Q_k(t,t_\xi,\xi) \to \mathcal Q_k(\infty,t_\xi,\xi)$ locally uniform with respect to $\xi\ne0$ as well as the invertibility of 
$\mathcal Q_k(\infty,t_\xi,\xi)$ based on the estimate for its determinant
\begin{multline}
    \det \mathcal Q_k(\infty,t_\xi,\xi) = \exp\left(\j\int_{t_\xi}^\infty \mathrm{trace}\, \widetilde{\mathcal R}_k(\tau,t_\xi,\xi) \d\tau\right)
   \\ \ge \exp\left(- d\int_{t_\xi}^\infty \|\widetilde{\mathcal R}_k(\tau,t_\xi,\xi)\| \d\tau \right).
\end{multline}
The fundamental solution to the original system is obtained from the diagonalisation procedure. Indeed, tracing back the transformations yields \eqref{eq:Eprodform}.
It remains to prove the symbolic estimate for the matrix $\mathcal Q_k(t,t_\xi,\xi)$. This follows by differentiating the series representation \eqref{eq:PeanoBaker} term by term. To establish bounds for derivatives of $\mathcal R_k(t,s,\xi)$, we first
observe that for $|\alpha|=1$
\begin{equation}
   \D_\xi^\alpha \widetilde{\mathcal E}_k (t,s,\xi) = \widetilde{\mathcal E}_k (t,s,\xi) \, \int_s^t \partial_\xi^\alpha\big(\mathcal D(\tau,\xi)+F_{k-1}(\tau,\xi)\big) \d\tau,
\end{equation}
where the integral on the right is bounded by $(t-s)$. Forming further derivatives and using Leibniz rule implies 
\begin{equation}
   \| \D_\xi^\alpha  \widetilde{\mathcal E}_k (t,s,\xi)\| \lesssim_{N,k,\alpha} \left(\frac{1+t}{1+s}\right)^{-\kappa_+} (t-s)^{|\alpha|},\qquad t\ge s,
\end{equation}
for all multi-indices $\alpha$. Derivatives with respect to $s$ yield
\begin{equation}
   \D_s \widetilde{\mathcal E}_k (t,s,\xi) = -\widetilde{\mathcal E}_k (t,s,\xi)  \big(\mathcal D(s,\xi)+F_{k-1}(s,\xi)\big) 
\end{equation}
and thus are estimated by multiplications by $|\xi|$. Again using Leibniz rule and combining it with the derivatives with respect to $\xi$ yields
\begin{equation}\label{eq:3.41}
   \| \D_s^\ell\D_\xi^\alpha  \widetilde{\mathcal E}_k (t,s,\xi)\| \lesssim_{N,k,\alpha,\ell} \left(\frac{1+t}{1+s}\right)^{-\kappa_+} (t-s)^{|\alpha|} |\xi|^\ell,\qquad t\ge s.
\end{equation}
Similarly we obtain 
\begin{equation}\label{eq:3.42}
   \| \D_s^\ell\D_\xi^\alpha  \widetilde{\mathcal E}_k (s,t,\xi)\| \lesssim_{N,k,\alpha,\ell} \left(\frac{1+t}{1+s}\right)^{\kappa_-} (t-s)^{|\alpha|} |\xi|^\ell,\qquad t\ge s.
\end{equation}
and hence estimates for derivatives of $\widetilde{\mathcal R}_k (t,t_\xi,\xi)$ follow from the symbolic behaviour of $R_k(t,\xi)$ combined with the estimates
\eqref{eq:3.41} and \eqref{eq:3.42} and read as
\begin{equation}
   \| \D_\xi^\alpha \widetilde{\mathcal R}_k(t,t_\xi,\xi) \| \lesssim_{N,k,\alpha} \left(\frac{1+t}{1+t_\xi}\right)^{\kappa_--\kappa_+} \frac{1}{(1+t)^{k+1} |\xi|^k} \left( (1+t)^{|\alpha|} + |\xi|^{-|\alpha|} \right).
\end{equation}
Hence, differentiating \eqref{eq:PeanoBaker} term by term with respect to $\xi$ yields in combination with \eqref{eq:3:txi-der-est}
\begin{equation}
  \| \D_\xi^\alpha \mathcal Q_k(t,t_\xi,\xi) \|  \lesssim_{N,k,\alpha} |\xi|^{-|\alpha|},\qquad t\ge t_\xi,
\end{equation}
as long as $|\alpha|\le k-1-\kappa_-+\kappa_+$ in order to guarantee uniform integrability
of the appearing terms in the series and thus uniformity of the estimate with respect to $t$. As there are only finitely many multi-indices involved, the constants in the estimate can be chosen independent of $\alpha$.
\end{proof}

\section{Energy and dispersive type estimates}\label{sec4}

\subsection{Energy estimates} Estimates for the $L^2$ norm of solutions follow directly from Theorems~\ref{thm:3.1} and~\ref{thm:3.5}. Let $\mu$ be defined by \eqref{eq:mu-def} 
and let $\kappa_+$ and $\kappa_-$ denote the constants from (A2). If $\sigma=1$ then the estimate 
\begin{equation}
   \|\mathcal E(t,0,\xi)\| \lesssim (1+t)^{-\min\{ \kappa_+, \mu\} }
\end{equation}
holds true and therefore any solution $U(t,x)$ of \eqref{eq:hyp-sys-1} satisfies the norm estimate
\begin{equation}
  \| U(t,\cdot)\|_2 \lesssim  (1+t)^{-\min\{ \kappa_+, \mu\} } \|U_0\|_2.
\end{equation}
If $\sigma>1$ we distinguish two cases. Either $\kappa_+<\mu$, then we obtain similarly to the above situation
\begin{equation}
  \| U(t,\cdot)\|_2 \lesssim (1+t)^{-\kappa_+ } \|U_0\|_2
\end{equation}
while for $ \kappa_+\ge \mu$
\begin{equation}
  \| U(t,\cdot)\|_2 \lesssim_\epsilon (1+t)^{-\mu+\epsilon } \|U_0\|_2
\end{equation}
holds true for any $\epsilon>0$.

\begin{rem}
Note that in contrast to results based on the GECL property introduced in  \cite{HW09}  where one step of diagonalisation in the hyperbolic zone is sufficient to derive energy estimates and hence energy estimates essentially follow under a $C^1$ assumption for the coefficients, our energy estimates need a higher amount of smoothness depending on $\kappa_--\kappa_+$.
\end{rem}

\subsection{Dispersive type estimates} For dispersive type estimates the behaviour of high frequencies is of importance. First, we recall some basic facts and notations. Associated to the eigenvalues $\lambda_1(t,\xi)$, \ldots, $\lambda_d(t,\xi)$ of the hyperbolic principal symbol $A_{\rm hom}(t,\xi)$ we define the auxiliary functions
\begin{equation}\label{eq:4.5}
\vartheta_j(t,\xi) = \frac1t \int_0^t \lambda_j(\tau,\xi) \d \tau
\end{equation}
and families of slowness surfaces
\begin{equation}
  \Sigma_t^{(j)} = \{ \xi\in\R^n : \vartheta_j(t,\xi) = 1 \}.
\end{equation}
For fixed $t$ these surfaces are smooth (even algebraic if $A(t,\xi)$ is polynomial with respect to $\xi$) and disjoint. 
They determine the dispersive properties of the Fourier integral operators 
\begin{equation}
   u_0 \mapsto \int \mathrm e^{\j ( x\cdot\xi + t\vartheta_j(t,\xi))}a(t,\xi) \widehat u_0(\xi) \d\xi
\end{equation}
for given amplitudes supported within $\mathcal Z_{\rm hyp}(N)$ and satisfying
\begin{equation}
   |\D_\xi^\alpha a(t,\xi) | \lesssim_\alpha  |\xi|^{-|\alpha|},\qquad t\ge t_\xi.
\end{equation}
Details on these estimates can be found in \cite{RW14} and will be recalled below.

\subsubsection{Assuming convexity} For the following we assume that the surfaces $\Sigma_t^{(j)}$ are strictly convex for sufficiently large  $t$. We recall the notion of asymptotic contact indices for surfaces from \cite[Chapter 4.7]{RW14} and define
\begin{equation}
   \gamma_{\rm as} (\Sigma_t^{(j)}; t\to\infty) = \min\{ \gamma\ge 2 :  \liminf_{t\to\infty} \varkappa(\Sigma_t^{(j)},\gamma)>0 \},
\end{equation}
where
\begin{equation} 
   \varkappa(\Sigma,\gamma) = \min_{p\in\Sigma} \varkappa(\Sigma,\gamma;p) 
\end{equation}
and $\varkappa(\Sigma,\gamma;p)$ is defined locally around points $p$ of $\Sigma$ as follows. By translation and rotation we may assume that $p$ is the origin and $\Sigma$ is given as graph $(y,h(y))$
of a function $h:\Omega\to R$ and with $h(0)=0$ and $h'(0)=0$. Then 
\begin{equation}
\varkappa(\Sigma,\gamma;p) = \inf_{|\omega|=1} \sum_{j=2}^\gamma \big|\partial_\rho^j h(\rho\omega)|_{\rho=0} \big|.
\end{equation}
Lower bounds on $\varkappa(\Sigma,\gamma;p)$ are quantitative measures of contact of order $\gamma$ between $\Sigma$ and its tangent plane in $p$. They have been introduced by the author and M.~Ruzhansky in \cite{RW11} and \cite{R12} based on earlier works of M.~Sugimoto \cite{Sug94}.

\begin{thm}
Assume (A1) to (A4) and that all the slowness surfaces $\Sigma_t^{(j)}$ are strictly convex for large $t$ and denote $\gamma=\min_j \gamma_{\rm as}(\Sigma^{(j)}_t; t\to\infty)$. 
Then for $1<p\le2$, $pq=p+q$, $\epsilon>0$ and with regularity $r>(n-\frac{n-1}\gamma) (\frac1p-\frac1q)$ the following estimates hold true:
\begin{enumerate}
\item If $\mu>\kappa_+$ then 
\begin{equation}
  \| U(t,\cdot)\|_q \lesssim_{p,\epsilon} (1+t)^{-\kappa_+ +\epsilon - \frac{n-1}\gamma (\frac1p-\frac1q)} \|U_0\|_{p,r}.
\end{equation}
If in addition $F_0$ is independent of $\xi$ then the estimate is valid for $\epsilon=0$. 
\item If $\mu\le\kappa_+$ then
\begin{equation}
  \| U(t,\cdot)\|_q \lesssim_{p,\epsilon} (1+t)^{-\mu+\epsilon - \frac{n-1}\gamma (\frac1p-\frac1q)} \|U_0\|_{p,r}.
\end{equation}
If in addition $\sigma=1$ in (A3) and $F_0$ is independent of $\xi$ then the estimate is valid for $\epsilon=0$.
\end{enumerate}
The constants $\kappa_\pm$ are determined by (A2), the constant $\mu$ by \eqref{eq:mu-def} from (A3).
\end{thm}
\begin{proof} We show the first estimate.
The solution $U$ to \eqref{eq:hyp-sys-1} are represented as $\widehat U(t,\xi) = \mathcal E(t,0,\xi) \widehat U_0(\xi)$. We decompose $U$ into two parts by means of two smooth cut-off functions $\chi_{\rm pd}(t,\xi) = \chi(N^{-1}(1+t)|\xi|)$ and $\chi_{\rm hyp} = 1 - \chi_{\rm pd}$, where $\chi\in C_0^\infty(\R)$ is non-negative, bounded by $1$ and satisfies $\chi(s)=1$ for $|s|\le 1$ and $\chi(s)=0$ for $|s|\ge2$. Then 
\begin{equation}
\begin{split}
  \| \chi_{\rm pd}(t,\cdot) \widehat U(t,\cdot) \|_{1}  &= \int  \chi_{\rm pd} (t,\xi) \| \widehat U(t,\xi)\| \d\xi \\& \le C (1+t)^{-\mu+\epsilon} \int_{(1+t)|\xi|\le 2N}  \|\widehat U_0(\xi)\| \d\xi \\
  &\le  C (1+t)^{-\mu+\epsilon-n} \|\widehat U_0\|_\infty
\end{split}
\end{equation}
based on the support of $\chi_{\rm pd}$ and the pointwise estimate of $\mathcal E(t,0,\xi)$ from Theorem~\ref{thm:3.1}. In consequence we obtain
\begin{equation}\label{eq:4.14}
   \| \chi_{\rm pd}(t,\D_x) U(t,\cdot)\|_\infty \le  C (1+t)^{-\mu+\epsilon-n} \|U_0\|_1.
\end{equation}
The treatment of $\chi_{\rm hyp}(t,\D_x)U$ is more subtle. Theorem~\ref{thm:3.5} together with  Proposition~\ref{prop:3.3} and Proposition~\ref{prop:3.4} implies the representation
\begin{equation}
    \chi_{\rm hyp}(t,\xi) \widehat U(t,\xi) = (1+t)^{-\kappa_+} \sum_{j=1}^d \mathrm e^{\j t\vartheta_j(t,\xi)} B_j(t,\xi) \widehat U_0(t,\xi)
\end{equation}
with the phase functions $\vartheta_j(t,\xi)$ from \eqref{eq:4.5} and matrices $B_j(t,\xi)$ arising as products of the diagonalisers, of $\mathcal Q_k(t,t_\xi,\xi)$,
the exponentials $\exp(\int_{t_\xi}^t F_{k-1}(\tau,\xi)\d\tau)$ and $\mathcal E(t_\xi,0,\xi)$. They   satisfy
\begin{equation}
   \| \D_\xi^\alpha B_j(t,\xi) \| \le C_\alpha |\xi|^{-|\alpha|} (\log(\mathrm e+t))^{|\alpha|}.
\end{equation}
The logarithmic term appears from derivatives of $\exp(\int_{t_\xi}^t F_0(\tau,\xi)\d\tau)$ and disappears if $F_0$ is independent of $\xi$. 

Finally, \cite[Theorem 4.7]{RW14} yields the desired decay estimates
\begin{equation}
\|\chi_{\rm hyp}(t,\D_x) U(t,\cdot)\|_\infty \le C (1+t)^{-\kappa_+ +\epsilon - \frac{n-1}\gamma} \|U_0\|_{B^r_{1,2}}
\end{equation} 
in terms of a Besov norm of the initial data with regularity $r=n-\frac{n-1}\gamma$. Combination with \eqref{eq:4.14} and interpolation with the already shown energy estimate yields the desired statement.

The second statement follows similarly with one difference. Here the decay in the zone $\mathcal Z_{\rm pd}(N)$ is slower and thus determines the final result.
\end{proof}

The second part of the result, where the zone $\mathcal Z_{\rm pd}(N)$ determines the decay order, can be further improved by assuming moment and decay conditions on the data in analogy to the corresponding statement of \cite{NW15}.

\subsubsection{No convexity assumption} Without assuming convexity for the slowness surfaces $\Sigma_t^{(j)}$ the resulting dispersive properties are essentially one-dimensional. We recall the non-convex asymptotic contact indices
\begin{equation}
   \gamma_{0,\rm as} (\Sigma_t^{(j)}; t\to\infty) = \min\{ \gamma\ge 2 :  \liminf_{t\to\infty} \varkappa_0(\Sigma_t^{(j)},\gamma)>0 \}.
\end{equation}
where now
\begin{equation} 
   \varkappa_0(\Sigma,\gamma) = \min_{p\in\Sigma} \varkappa_0(\Sigma,\gamma;p),\qquad \varkappa_0(\Sigma,\gamma;p) = \sup_{|\omega|=1} \sum_{j=2}^\gamma \big|\partial_\rho^j h(\rho\omega)|_{\rho=0} \big|.
\end{equation}
The proof of the following statement is analogous to the convex situation just replacing the abstract result on Fourier integral operators for convex slowness surfaces by the corresponding result for general non-convex surfaces. Of course both situations can be mixed if some of the surfaces are convex. The overall decay is determined by the slowness surface implying the lowest decay rate.

\begin{thm}
Assume (A1) to (A4) and denote by $\gamma=\min_j \gamma_{0,\rm as}(\Sigma^{(j)}_t; t\to\infty)$ the minimal non-convex asymptotic contact index of the slowness surfaces
$\Sigma_t^{(j)}$. Then for $1<p\le2$, $pq=p+q$, $\epsilon>0$ and with regularity $r>(n-\frac1{\gamma_0}) (\frac1p-\frac1q)$ the following estimates hold true. 
\begin{enumerate}
\item 
If  $\mu >  \kappa_+$ then 
\begin{equation}
  \| U(t,\cdot)\|_q \lesssim_{p,\epsilon} (1+t)^{-\kappa_++\epsilon- \frac{1}{\gamma_0} (\frac1p-\frac1q)} \|U_0\|_{p,r}.
\end{equation}
If in addition $F_0$ is independent of $\xi$ then the estimate is also valid for $\epsilon=0$. 
\item 
If  $\mu \le  \kappa_+$ then 
\begin{equation}
  \| U(t,\cdot)\|_q \lesssim_{p,\epsilon} (1+t)^{-\mu+\epsilon - \frac{1}{\gamma_0} (\frac1p-\frac1q)} \|U_0\|_{p,r}.
\end{equation}
If in addition $\sigma=1$ in (A3) and $F_0$ is independent of $\xi$ then the estimate is also valid for $\epsilon=0$. 
\end{enumerate}
The constants $\kappa_\pm$ are determined by (A2), the constant $\mu$ by \eqref{eq:mu-def} from (A3).
\end{thm}

\section{Coming back to our examples}\label{sec5}

In this final section we turn back to our motivating examples and provide some particular results following from the general theory. We restrict to specific examples to make the advantage of our approach more transparent.

\subsection{Differential symmetric hyperbolic systems} Let $A_j\in\mathcal T\{0\}\otimes\C^{d\times d}$ be self-adjoint matrices such that $A(t,\xi)$ defined by \eqref{eq:A-pol} has uniform in $t$ distinct real eigenvalues for all non-zero real $\xi$ and $B\in\mathcal T\{1\}\otimes\C^{d\times d}$ such that 
\begin{equation}\label{eq:5.1}
   \int_1^\infty \| tB(t) - B_\infty\|^\sigma \frac{\d t}{t} < \infty
\end{equation}
for some $B_\infty\in\C^{d\times d}$ and some number $\sigma\ge1$. Then assumptions (A1) to (A4) are clearly satisfied for
\begin{equation}\label{eq:5.2}
    \D_t U = \sum_{j=1}^n A_j(t) \D_{x_j} U + B(t) U,\qquad\qquad U(0,\cdot)=U_0,
\end{equation}
 and our theory is applicable.  It turns out that large-time estimates for solutions to \eqref{eq:5.2} depend mainly on the properties of $B_\infty$.  
 Indeed, as shown in \cite[Lemma~4.2]{RW14} and \cite[Prop.~4 and Remark~5]{RW11}, there exists a family $M(t,\xi)$ of unitary matrices diagonalising $A(t,\xi)$, satisfying $M,M^*\in\mathcal S\{0,0\}$ and $\Im\diag((\D_t M^*)M)=0$. Hence, modulo integrable terms we obtain in the hyperbolic zone $\mathcal Z_{\rm hyp}(N)$ the representation $F_0(t,\xi) = t^{-1}\diag M^*(t,\xi)B_\infty M(t,\xi)$ and therefore the diagonal entries of $t F_0(t,\xi)$ must be elements of the numerical range
\begin{equation}
    \mathcal W(B_\infty) = \{ v^* B_\infty v : v\in\C^{d}, \|v\|=1\}
\end{equation}
of the matrix $B_\infty$. Hence, using
\begin{equation} 
\kappa_+ = \min \Im \mathcal W(B_\infty) = \min \spec \Im B_\infty
\end{equation} 
and 
\begin{equation}
\kappa_- = \max \Im \mathcal W(B_\infty) = \max \spec \Im B_\infty
\end{equation}
we obtain the estimates
\begin{equation}
  \left\| \exp\left(\j \int_s^t F_0(\tau,\xi) \d\tau\right)\right\| \lesssim_N \left(\frac{1+t}{1+s}\right)^{-\kappa_+},\qquad t\ge s\ge t_\xi,
\end{equation}
and
\begin{equation}
  \left\| \exp\left(\j \int_s^t F_0(\tau,\xi) \d\tau\right)\right\| \lesssim_N \left(\frac{1+t}{1+s}\right)^{\kappa_-},\qquad s\ge t\ge t_\xi.
\end{equation}
This yields the corresponding estimate for the fundamental solution $\mathcal E(t,s,\xi)$ in the hyperbolic zone $\mathcal Z_{\rm hyp}(N)$. In the pseudo-differential zone $\mathcal Z_{\rm pd}(N)$ the hyperbolic principal symbol $A_{\rm hom}(t,\xi)$ plays no role as
\begin{equation}
  \int_1^{t_\xi} \| tA_{\rm hom}(t,\xi)\| \frac{\d t}{t} \le c |\xi| \int_1^{t_\xi} \d t \le c |\xi| t_\xi \le c N
\end{equation}
and (A3) follows directly from \eqref{eq:5.1}. Hence, we obtain within $\mathcal Z_{\rm pd}(N)$ for $\sigma=1$
\begin{equation}
   \|\mathcal E(t,s,\xi)\|\lesssim_N \left(\frac{1+t}{1+s}\right)^{-\mu}
\end{equation}
with $\mu = \min\Im\spec B_\infty \ge \kappa_+$ such that the resulting estimate is an energy decay estimate with rate $t^{-\kappa_+}$. 
For $\sigma>1$ and $\mu=\kappa_+$ we lose a small amount of decay, otherwise the same estimate holds true.

Dispersive estimates depend in the behaviour of the matrices $A_j(t)$. For the following we assume that the limits $\lim_{t\to\infty} A_j(t) = A_j(\infty)$ exist. This implies that the (by assumption uniformly distinct and real) eigenvalues $\lambda_k(t,\xi)$ of $A(t,\xi) = \sum_j A_j(t)\xi_j$ converge to the eigenvalues $\lambda_k(\xi)$ of $A(\xi)=\sum_j A_j(\infty)\xi_j$ and therefore
\begin{equation}\label{eq:5.5}
   \vartheta_k(t,\xi) = \frac1t \int_0^t \lambda_k(\tau,\xi)\d\tau \to \lambda_k(\xi),\qquad t\to\infty.
\end{equation}
This implies that the slowness surfaces $\Sigma_t^{(k)}$ converge to the surfaces
\begin{equation}
	\Sigma^{(k)} = \{\xi\in\R^n : \lambda_k(\xi)=1\}.
\end{equation}
The convergence in \eqref{eq:5.5} is locally uniform for all $\xi$-derivatives as long as $\xi\ne0$. Therefore the surfaces $\Sigma_t^{(k)}$ converge in $C^\infty$ 
and in order to apply the results of \cite[Chapter 4.7]{RW14} it suffices to consider the surfaces $\Sigma^{(k)}$. 

\begin{thm}
Assume all the surfaces $\Sigma^{(k)}$ are strictly convex and let
\begin{equation}
  \gamma = \max_{k=1,\ldots,d} \gamma(\Sigma^{(k)}). 
\end{equation}
Then solutions to \eqref{eq:5.2} satisfy the estimate
\begin{equation}
   \| U(t,\cdot)\|_q \lesssim_{p,\epsilon} (1+t)^{-\kappa_++\epsilon-\frac{n-1}\gamma (\frac1p-\frac1q)} \|U_0\|_{p,r} 
\end{equation}
for $1<p\le 2$, $pq=p+q$ and regularity $r>(n-\frac{n-1}\gamma) (\frac 1p-\frac1q)$.
 If $\sigma=1$ or $\kappa_+<\mu$ we may choose $\epsilon=0$, otherwise the estimate holds true for arbitrary $\epsilon>0$.
\end{thm}

\subsection{A particular symmetric hyperbolic system}
On $\R^n$ we consider the particular symmetric hyperbolic system
\begin{equation}
 \D_t U = \left( \begin{pmatrix} & |\D_x| \\ |\D_x| & \end{pmatrix} + \frac{\mathrm i}{1+t} \begin{pmatrix} b_1(t) & \\ & b_2(t) \end{pmatrix} \right) U
\end{equation}
with real-valued coefficients $b_1,b_2\in\mathcal T\{0\}$. Examples of such coefficients could be oscillating like $b_j(t) = \sin(\alpha_j \log(t)+\beta_j) + \gamma_j$ for coefficients $\alpha_j,\beta_j,\gamma_j\in\R$, $j=1,2$ and are therefore more general as the ones just considered. We assume that 
\begin{equation}\label{eq:5.2:weakdich}
  t\mapsto  \int_1^t \big(b_1(\tau) - b_2(\tau) \big) \frac{\d\tau}{\tau}  
\end{equation}
is uniformly in $t\ge1$ bounded from below or from above. Then for this system assumptions (A1) to (A4) are all satisfied. Indeed, with
\begin{equation}
   A(t,\xi) = \begin{pmatrix}  & |\xi| \\ |\xi| & 0 \end{pmatrix} + \frac{\mathrm i}{1+t} \begin{pmatrix} b_1(t) & \\ & b_2(t) \end{pmatrix} \in \mathcal S\{1,0\} + \mathcal S\{0,1\}
\end{equation}
and the fact that the principal part is independent of $t$ we see that in (A2) the hyperbolic subprincipal symbol is independent of $\xi$. It is given by 
\begin{equation}
F_0(t,\xi)=\frac{\mathrm i}{1+t} \frac{b_1(t)+b_2(t)}2 \mathrm I.
\end{equation} 
Furthermore, in (A3) the diagonaliser $\tilde M$ can be chosen as identity matrix and the large-time principal part is determined as $\Lambda(t) = \mathrm i (1+t)^{-1} \diag(b_1(t),b_2(t))$.
The weak dichotomy condition is equivalent to the one-sided boundedness of \eqref{eq:5.2:weakdich}. Thus (A3) holds true with $\sigma=1$. 

The decay rate in the zone $\mathcal Z_{\rm pd}(N)$ is thus determined by the minimum of the numbers
\begin{equation}
   \mu_j = \liminf_{t\to\infty} \frac{\int_1^t b_j(\tau)\frac{\d\tau}{\tau}} {\log t}  \in\bigg[ \liminf_{t\to\infty} b_j(t), \limsup_{t\to\infty} b_j(t) \bigg], 
\end{equation}
while the decay in $\mathcal Z_{\rm hyp}(N)$ is determined by their arithmetic mean $(\mu_1+\mu_2)/2$ and thus stronger. If we denote $\mu=\min\{\mu_1,\mu_2\}$ then the main result of this paper implies the dispersive type estimate
\begin{equation}
   \| U(t,\cdot)\|_q \lesssim_p (1+t)^{-\mu - \frac{n-1}2 (\frac 1p-\frac1q)} \| U_0 \|_{p,r}
\end{equation}
for $1<p\le 2$, $pq=p+q$ and $r> \frac{n+1}2 (\frac1p-\frac1q)$.

\subsection{A wave equation with weak dissipation}
Wave equations with weak and scale invariant dissipation 
\begin{equation}
   u_{tt} - \Delta u + b(t) u_t = 0,\qquad  u(0,\cdot) = u_0,\quad u_t(0,\cdot)=u_1
\end{equation}
have been considered by the author in \cite{Wir04} and \cite{Wir06}. The approach of \cite{Wir06} allows for real-valued $b\in\mathcal T\{1\}$ such that
\begin{equation}
\limsup_{t\to\infty} tb(t) < 1
\end{equation}
holds true. The present paper allows to overcome this restriction and to consider {\em arbitrary} non-negative $b\in\mathcal T\{1\}$ and to derive dispersive type estimates for solutions. Of interest is the number
\begin{equation}
 \mu =  \liminf_{t\to\infty} \frac{\int_1^t b(\tau)\d\tau}{\log t} 
\end{equation}
and considerations boil down to the question whether $\mu\le 2$ or $\mu>2$. 

Rewriting the problem as system of first order in $\widehat U=(|\xi| \widehat u, \D_t \widehat u)^\top$, we obtain 
\begin{equation}
 A(t,\xi) = \begin{pmatrix} & |\xi| \\ |\xi| & \end{pmatrix} + \mathrm i b(t) \begin{pmatrix} 0 & 0 \\ 0  & 1\end{pmatrix} \in \mathcal S\{1,0\} + \mathcal S\{0,1\}
\end{equation}
and assumptions (A1) to (A4) are satisfied with $F_0(t,\xi) = \mathrm i b(t) /2 \mathrm I$ independent of $\xi$, $\tilde M =\mathrm I$ and for $\sigma=1$.  
The weak dichotomy condition is clearly satisfied for $b\ge0$ and the behaviour in $\mathcal Z_{\rm pd}(N)$ is determined from the numbers $0$ and $\mu$, 
while in $\mathcal Z_{\rm hyp}(N)$ the overall decay stems from $\mu/2$. 

We are not looking for estimates of $U$ in terms of $U_0$, but rather of norms for $\nabla u$ and $\partial_t u$ in terms of Sobolev norms for $u_0$ and $u_1$. This gives for the first component and small frequencies a further factor or $|\xi|$ on the Fourier side (cf. formula \eqref{eq:3.11} for $t\ge1$)
\begin{equation}
 \underbrace{ \begin{pmatrix}  \mathcal O(1) & \mathcal O(t^{-\mu}) \\   \mathcal O(1) & \mathcal O(t^{-\mu})  \end{pmatrix}}_{=(V_1(t,\xi)  | V_2(t,\xi))}
 \underbrace{ \begin{pmatrix}  \mathcal O(1) & \mathcal O(|\xi|) \\   \mathcal O(|\xi|) & \mathcal O(1)  \end{pmatrix} }_{=(V_1(1,\xi)  | V_2(1,\xi))^{-1}}
        \begin{pmatrix} |\xi| / \langle\xi\rangle \\ 1 \end{pmatrix}\qquad \text{in $\mathcal Z_{\rm pd}(N)$}
\end{equation}
and hence additional decay of one order in the first component.  We distinguish two cases. If $\mu\le2$, then the estimate
\begin{equation}
   \| \nabla u(t,\cdot)\|_q + \|\partial_t u(t,\cdot)\|_q \lesssim_q (1+t)^{-\frac\mu2 - \frac{n-1}2 (\frac1p-\frac1q)} \big(\|u_0\|_{p,r+1} + \|u_1\|_{p,r}\big)
\end{equation}
holds true for $1<p\le 2$, $pq=p+q$ and with $r=\frac{n+1}2(\frac1p+\frac1q)$. The decay is determined by the behaviour of large frequencies.
If $\mu>2$, then the estimate
\begin{equation}
   \| \nabla u(t,\cdot)\|_q + \|\partial_t u(t,\cdot)\|_q \lesssim_q (1+t)^{-1 - \frac{n-1}2 (\frac1p-\frac1q)} \big(\|u_0\|_{p,r+1} + \|u_1\|_{p,r}\big)
\end{equation}
holds true under the same conditions and the decay is determined by the behaviour of small frequencies.

\subsection{Klein--Gordon equations with time-dependent potential}
Next we consider Klein--Gordon equations with $t$-dependent mass term,
\begin{equation}\label{eq:5.2.1}
   u_{tt} - \Delta u + \frac{m^2(t)}{(1+t)^2} u = 0,\qquad u(0,\cdot) = u_0,\quad u_t(0,\cdot)=u_1.
\end{equation}
We assume that $m\in\mathcal T\{0\}$ is strictly non-negative. The scale-invariant case $m(t) = m_0$  has been treated by C.~B\"ohme and M.~Reissig in \cite{BR12} by means of special functions. Faster decaying coefficients with $(t\mapsto m(t)/(1+t)) \in\mathrm L^1(\R_+)$ appear in C.~B\"ohme's PhD thesis \cite{B11}, while coefficients with slower decay rate appear in the treatment of \cite{BR13}. Our approach allows us to provide results for perturbations of the scale-invariant case. 

To be precise, we assume that $m\in\mathcal T\{0\}$ satisfies
\begin{equation}
     \int_1^\infty | m(t) - m_0 |^\sigma \frac{\d t}{t} < \infty 
\end{equation}  
for some $m_0>0$ and a number $\sigma\ge1$. As example we could consider 
\begin{equation}
  m(t) =  m_0 + \sum_{j=0}^\infty \alpha_j \psi(2^{-j} t) 
\end{equation}
for some $\psi\in\mathrm C_c^\infty(\R)$ with $\mathrm{supp}\, \chi\subset[1,2]$ subject to $\int_1^2 \psi(s)\frac{\d s}{s}=1$ and a sequence $(\alpha_j)\in \ell^\sigma$. 
We apply a partial Fourier transform and rewrite \eqref{eq:5.2.1} as system of first order in the unknown
\begin{equation}
\widehat  U(t,\xi) = \begin{pmatrix}
      \sqrt{|\xi|^2+1/(1+t)^2} \,\, \widehat u(t,\xi) \\ \D_t \widehat u(t,\xi)
  \end{pmatrix}
\end{equation} 
such that
\begin{equation}
  \D_t \widehat U(t,\xi) = \begin{pmatrix} \frac{\mathrm i / (1+t)}{(1+t)^2|\xi|^2 + 1} & \sqrt{|\xi|^2+1/(1+t)^2} \\\\ \frac{|\xi|^2+m^2(t)}{\sqrt{|\xi|^2+1/(1+t)^2}} &  \end{pmatrix}\widehat U(t,\xi)  = A(t,\xi) \widehat U(t,\xi).
\end{equation}
The coefficient matrix $A(t,\xi)$ belongs to $\mathcal S\{1,0\}$. This follows from
\begin{equation}
\frac{|\xi|^2+m^2(t)}{\sqrt{|\xi|^2+1/(1+t)^2}} ,\; \sqrt{|\xi|^2+1/(1+t)^2} \in\mathcal S\{1,0\}
 \end{equation}
 and 
\begin{equation}\label{eq:5/.6}
   \frac{\mathrm i / (1+t)}{(1+t)^2|\xi|^2 + 1}   \in \mathcal S\{-2,3\}\subset\mathcal S\{0,1\}.
\end{equation}
Using that
\begin{equation}\label{eq:5/.7}
  \sqrt{|\xi|^2+1/(1+t)^2} - |\xi| , \;  \frac{|\xi|^2+m^2(t)}{\sqrt{|\xi|^2+1/(1+t)^2}} -|\xi| \in \mathcal S\{-1,2\}\subset \mathcal S\{0,1\} 
\end{equation}
we see that the hyperbolic principal part is given by
\begin{equation}
  A_{\rm hom}(t,\xi) = \begin{pmatrix} & |\xi| \\ |\xi| \end{pmatrix}
\end{equation} 
and thus \eqref{eq:5/.6} and \eqref{eq:5/.7} imply that $\kappa_\pm=0$ in (A2). Assumption (A3) is satisfied with constant large-time principal symbol if we use
\begin{equation}
A_\infty = \begin{pmatrix} \mathrm i & 1 \\ m_0 ^2 & 0 \end{pmatrix}
\end{equation}
and it is easy to check that (A4) also holds true for any number $N$.  In particular we see that $\spec A_\infty = \{ \mathrm i/2 \pm \sqrt{m_0^2-1/4} \}$.

The characteristic roots $\lambda_\pm(t,\xi) = \pm |\xi|$ are independent of $t$ and arise from the wave equation as principal part. Thus we obtain with
\begin{equation}
 \mu=\begin{cases} 
 \frac12, \quad &  m_0\ge \frac12,\\
\frac12+\sqrt{\frac14-m_0^2}, \quad &0<m_0<\frac 12
\end{cases}
\end{equation}
the dispersive type estimate 
\begin{multline}
   \| u_t (t,\cdot)\|_q + \| \sqrt{\Delta+1/(1+t)^2} u(t,\cdot)\|_q 
   \\
   \lesssim_{m,p,\epsilon} (1+t)^{\epsilon-\mu -\frac {n-1}2\left(\frac1p-\frac1q\right)} \big(\|u_0\|_{p,r+1} + \|u_1\|_{p,r}\big)
\end{multline}
for $1<p\le 2$, $pq=p+q$ and $r>\frac{n+1}2(\frac1p-\frac1q)$ generalising the results from \cite{BR12}.

\subsection{Models in cosmology} Equations with time-dependent coefficients arise in a natural way when studying the physics of an expanding universe. This has attracted considerable interest over the recent years, just to mention a few references we refer to the work of K.~Yagdjian, A.~Galstian and T.~Kinoshita \cite{GKY10, Yag13, GY15} and sketch the relation to our approach. Especially for the family of Friedmann--Lema\^itre--Robertson--Walker spacetimes on $\R^{1+n}$ with metrics of the form 
\begin{equation}
\d s^2 = - \d t^2 + \frac1{a^2(t)} \sum_{j=1}^n \d x_j^2,\qquad a(t) >0,
\end{equation}
in an expanding / shrinking universe the covariant Klein--Gordon equation for scalar fields takes the form
\begin{equation}
    u_{tt} - a^2(t) \Delta u -  n \frac{\dot a(t)}{a(t)} u_t + m_0 u = 0
\end{equation}
Here and later on $\dot a$ denotes the $t$-derivative of $a$. Particular cases contain the de Sitter model with $a(t) = \mathrm e^{-t}$, the Einstein--de Sitter model with $a(t)=(1+t)^{-2/3}$ or the anti-de Sitter model with $a(t)=\mathrm e^t$. A partial Liouville transform allows to reduce the 
equation to constant coefficients in the principal part. Indeed, let 
\begin{equation}
    A(t) = \int_0^t a(\theta) \d \theta 
\end{equation}
and denote by $\tau=A(t)$ the new time variable. Then a simple calculation yields for the new unknown $\widetilde u(\tau,x) = u(A^{-1}(\tau),x)$
%\begin{equation}
%  a^2(t) \widetilde u_{\tau\tau} - a^2(t) \Delta \widetilde u + (2- n) {\dot a(t)} \widetilde u_{\tau}  + m_0 \widetilde u = 0,
%\end{equation}
%where $t=A^{-1}(\tau)$. Dividing by $a^2(t)$ yields 
the equation
\begin{equation}
   \widetilde u_{\tau\tau} - \Delta \widetilde u + b(\tau)\widetilde u_\tau + m(\tau) \widetilde u = 0,\qquad \tau\in[0,T],
\end{equation}
with variable mass and dissipation. The time $T$ is given by $T=\int_0^\infty a(\theta)\d\theta$ and may be finite. We are interested in the case when $T=\infty$ and thus excluding the de Sitter model. Such equations were considered in \cite{NW15} and we will only comment on the assumptions to be made on the coefficients. The dissipation term is given by
\begin{equation}
    b(\tau) =  (2-n) \frac{\dot a(A^{-1}(\tau))}{a^2(A^{-1}(\tau))}  
\end{equation}
and the mass term 
\begin{equation}
     m(\tau) =  \frac{m_0}{a^2(A^{-1}(\tau))}.
\end{equation}
In order for our approach to be applicable, we require that $b\in\mathcal T\{1\}$ and that $m\in\mathcal T\{2\}$. This can be expressed in terms of estimates for the coefficient function $a(t)$, its derivatives and the primitive $A(t)$. Indeed,  the estimate $|b(\tau)|\le C (1+\tau)^{-1}$ is equivalent to the requirement
\begin{equation}
    \frac{|\dot a (t)|}{a(t)} \le C \frac{a(t)}{1+A(t)},
\end{equation}
similarly $m(\tau)\le C (1+\tau)^{-2}$ reduces to 
\begin{equation}
 C\le   \frac{a(t)}{1+A(t)} \qquad\text{or}\qquad m_0=0.
\end{equation}
In order to obtain all estimates for derivatives we require in addition to these two estimates that all higher order derivatives satisfy
\begin{equation}
|  \partial_t^k a(t) | \le C_k a(t) \left( \frac{a(t)}{1+A(t)}\right)^k.
\end{equation}
This corresponds (in analogy) to the assumptions in \cite{RY00b} and immediately implies $b\in\mathcal T\{1\}$ and $m\in\mathcal T\{2\}$. Such estimates are satisfied 
both for the Einstein--de Sitter and the anti-de Sitter model. Indeed if $a(t) = (1+t)^\ell$ for some number $\ell>-1$ then
\begin{equation}
   A(t)  = \int_0^t a(\theta)\d\theta = \frac1{\ell+1} \left((1+t)^{\ell+1} - 1 \right),\qquad t=\left(1+(\ell+1)\tau\right)^{\frac1{\ell+1}},
\end{equation}
and hence
\begin{equation}
   b(\tau)  = (2-n)\ell  \left(1+\big(1+(\ell+1)\tau\big)^{\frac1{\ell+1}}\right)^{-(\ell+1)} \sim (2-n) \frac{\ell}{\ell+1} \tau^{-1} + \mathcal O(\tau^{-2})
\end{equation}
as $\tau\to\infty$. The mass term does not have enough decay for our theory,
\begin{equation}
   m(\tau)  = m_0  \left(1+\big(1+(\ell+1)\tau\big)^{\frac1{\ell+1}}\right)^{-2\ell} \sim m_0(\ell+1)^{-2\frac{\ell}{\ell+1}}  \tau^{-\frac{2\ell}{\ell+1}} + \mathcal O(\tau^{-\frac{2\ell}{\ell+1}-1}), 
\end{equation}
so we need to require $m_0=0$ (or else, we have to consider equations with an effective mass term). Similarly, we obtain for the anti-de Sitter model with $a(t)=\mathrm e^{t}$ and hence $A(t) = \mathrm e^t-1$
\begin{equation}
   b(\tau)=\frac{2-n}{1+\tau},\qquad\qquad m(\tau) = \frac{m_0}{(1+\tau)^2},
\end{equation}
which both behave in the right way and satisfy the assumptions of our theory. For the de Sitter model the transformation yields $b(\tau)=(n-2)/(1-\tau)$ and
$m(\tau)=m_0 / (1-\tau)^2$, but restricted to the interval $[0,1]$.

\end{document}